\documentclass[oneside,english,reqno]{amsart}
\usepackage[T1]{fontenc}
\usepackage[latin9]{inputenc}
\usepackage{amsthm}
\makeatletter
\numberwithin{equation}{section} %% Comment out for sequentially-numbered
\numberwithin{figure}{section} %% Comment out for sequentially-numbered
\newtheorem{thm}{Theorem}[section]

\newtheorem{lem}[thm]{Lemma}
\newtheorem{prop}[thm]{Proposition}
\theoremstyle{definition}

\theoremstyle{remark}
\newtheorem{rem}[thm]{Remark}
\theoremstyle{example}

\numberwithin{equation}{section}
\newcommand{\numberset}{\mathbb}
\newcommand{\N}{\numberset{N}}

\makeatother
\usepackage{babel}
\begin{document}

\title[On Strong convergence of Halpern's method]{On Strong convergence of Halpern's method using  averaged type mappings}
\author{ F. Cianciaruso, G. Marino, A. Rugiano, B. Scardamaglia}%
\address{F. Cianciaruso, G. Marino, A. Rugiano, B. Scardamaglia: Dipartimento di Matematica ed Informatica,
Universit\'a della Calabria, 87036 Arcavacata di Rende (CS),
ITALY}
\email{ cianciaruso@unical.it, gmarino@unical.it}%
\email{ rugiano@mat.unical.it, scardamaglia@mat.unical.it}%
%
%\thanks{}%

%\date{}%
%\dedicatory{}%
%\commby{}%
% ----------------------------------------------------------------
\begin{abstract}
In this paper, inspired by Iemoto and Takahashi [S. Iemoto, W.
Takahashi, Nonlinear Analysis 71, (2009), 2082-2089], we study the Halpern's method to
approximate strongly fixed points of  a nonexpansive mapping and
of a nonspreading mapping. A crucial tool in our results is the
regularization with the averaged type mappings [C. Byrne, Inverse Probl. 20, (2004),
103-120].
\end{abstract}
\maketitle

\section{Introduction}
Let $H$ be a real Hilbert space with the inner product $\langle\cdot,\cdot\rangle$, which induces the norm $\|\cdot\|$.\\
Let $C$ be  a nonempty, closed and convex subset of $H$.
Let $T$ be a nonlinear mapping of $C$ into itself; we denote with $Fix(T)$ the set of fixed points of $T$, that is, $Fix(T)=\{z\in C: Tz=z\}$.\\
We recall that a mapping $T$ is said to be nonexpansive if
\begin{equation*}
\|Tx-Ty\|\leq \|x-y\|, \quad \forall x,y    \in C.
\end{equation*}

The problem of finding fixed points of nonexpansive mappings has been widely investigated by many authors.\\
For a fixed $u\in C$ and for each $t\in (0,1)$,  let $z_t$ be the unique fixed point of the contraction given by  $$T_tx=tu+(1-t)Tx,\quad x\in C.$$ Namely, we have $z_t=tu+(1-t)Tz_t$. Browder \cite{Browder} proved the following strong convergence theorem.
\begin{thm}\label{browder}
Let $C$ be a nonempty, bounded, closed  and convex subset of a Hilbert space $H$ and let $T: C\rightarrow C$ be a nonexpansive mapping. Fix $u\in C$ and define $z_t\in C$ as $z_t=tu+(1-t)Tz_t$ for $t\in (0,1)$. Then as $t$ tends to $0$, $z_t$ converges strongly to the unique element of $Fix(T)$ nearest to $u$, i.e.
$z_t$ converges strongly to $P_{Fix(T)}u$.
\end{thm}
We recall that if $C$ is a nonempty closed convex subset of $H$, then
for every point $x\in H$, there exists a unique nearest point in $C$, denoted by $P_Cx$, such that
$$\|x-P_Cx\|\leq\|x-y\|,\quad \forall y\in C.$$
Such $P_C$ is called the metric projection of $H$ onto $C$.\\
If $T$ is a nonexpansive mapping and  $u\in C$ fixed, Halpern \cite{Halpern} was the first who considered the following explicit method:

\begin{equation}\label{halpern}
x_1\in C, \quad x_{n+1}=\alpha_n u+(1-\alpha_n)Tx_n, \quad \forall n\geq1
\end{equation}
where $(\alpha_n)_{n\in\N}\subset [0,1]$.\\
Moreover, Halpern proved in \cite{Halpern} the following Theorem on the convergence of $(\ref{halpern})$ for a particular choice of  $(\alpha_n)_{n\in\N}$.
\begin{thm}
Let $C$ be a bounded, closed  and convex subset of a Hilbert space $H$ and let $T: C\rightarrow C$ be a nonexpansive mapping. For any initialization $x_1\in C$ and anchor $u\in C$, define a sequence $(x_n)_{n\in\N}$ in $C$ by
$$x_{n+1}=n^{-\theta}u+(1-n^{-\theta})Tx_n, \quad \forall n\geq1,$$
where $\theta\in (0,1)$.
Then $(x_n)_{n\in\N}$ converges strongly to the element of $Fix(T)$ nearest to $u$.
\end{thm}
He also showed that the control conditions
\begin{enumerate}
  \item[(C1)] $\displaystyle \lim_{n\rightarrow\infty}\alpha_n=0$;
  \item[(C2)] $\displaystyle\sum_{n=1}^\infty\alpha_n=\infty$,
\end{enumerate}
are necessary  for the convergence of  $(\ref{halpern})$ to a fixed point of $T$.\\
Subsequently, several authors carefully studied the following problem: are the control conditions $(C1)$ and $(C2)$ sufficient  for the convergence of $(\ref{halpern})$?\\
In this direction, C.E. Chidume and  C.O. Chidume
\cite{Chidume-Chidume} and Suzuki \cite{Suzuki}, independently,
proved that the  conditions $(C1)$ and $(C2)$ are sufficient to
assure the strong convergence to a fixed point of $T$ of the
following iterative sequence:
$$ x_1,u\in C; \quad x_{n+1}=\alpha_nu+(1-\alpha_n)(\lambda x_n+(1-\lambda)Tx_n),\quad \forall n\geq1.$$
Recently, in the setting of Banach spaces, Song and Chai \cite{Song-Chai}, under the same
conditions $(C1)$ and $(C2)$, but under stronger hypotheses on the mapping, obtained strong convergence of Halpern
iterations $(\ref{halpern})$. In particular, they assumed that $E$
is a real reflexive Banach space with a uniformly
Gate$\hat{a}$ux differentiable norm and with the fixed point property for nonexpansive self-mappings and considered an important subclass of nonexpansive mappings: the firmly type nonexpansive mappings.\\
Let $T$ be a mapping with domain $D(T)$. T is said to be firmly
type nonexpansive \cite{Song-Chai} if for all $x,y\in D(T)$, there
exists $k\in (0,+\infty)$ such that
\begin{equation*}
\|Tx-Ty\|^2\leq \|x-y\|^2-k\|(x-Tx)-(y-Ty)\|^2.
\end{equation*}

A more general class of firmly type nonexpansive mappings is the
class of the  strongly nonexpansive mappings. Recall that a
mapping $T: C\rightarrow C$ is said to be strongly nonexpansive
if:
\begin{enumerate}
  \item $T$ is nonexpansive;
  \item $x_n-y_n-(Tx_n-Ty_n)\rightarrow 0$, whenever $(x_n)_{n\in\N}$ and $(y_n)_{n\in\N}$ are sequences in $C$ such that $(x_n-y_n)_{n\in\N}$ is bounded and $\|x_n-y_n\|-\|Tx_n-Ty_n\|\rightarrow 0$.
\end{enumerate}
Saejung \cite{Saejung} proved the strong convergence of the
Halpern's iterations $(\ref{halpern})$ for strongly nonexpansive
mappings in a Banach space $E$ such that one of the following
conditions is satisfied:
\begin{itemize}
  \item $E$ is uniformly smooth;
  \item $E$ is reflexive, strictly convex with a uniformly Gate$\hat{a}$ux differentiable norm.
\end{itemize}

In the setting of Hilbert spaces, Kohsaka and Takahashi \cite{Kohsaka-Taka} defined
 $T: C\rightarrow
C$ a nonspreading mapping if:
\begin{equation*}
2\|Tx-Ty\|^2\leq \|Tx-y\|^2+\|x-Ty\|^2, \quad \forall x,y\in C.
\end{equation*}
The following Lemma is an useful characterization of a
nonspreading mapping.
\begin{lem}\cite{Iemoto-Taka}
Let $C$ be a nonempty closed subset of a Hilbert space $H$. Then a mapping $T: C\rightarrow C$ is nonspreading if and only if
\begin{equation}\label{S}
\|Tx-Ty\|^2\leq \|x-y\|^2+2\langle x-Tx, y-Ty\rangle, \quad \forall x,y\in C.
\end{equation}
\end{lem}
Observe that if $T$ is a nonspreading mapping from $C$ into itself
and $Fix(T)\neq\emptyset$, then $T$ is quasi-nonexpansive, i.e.
$$\|Tx-p\|\leq \|x-p\|,\quad \forall
x\in C, \quad \forall p\in Fix(T).$$

Further, the set of fixed points of a quasi-nonexpansive mapping is closed and convex \cite{Ito}.\\
Osilike and Isiogugu \cite{Osilike} studied the Halpern's type for $k-$strictly pseudononspreading mappings $T$, which are a  more general class of the nonspreading mappings.\\
To obtain the strong convergence of (\ref{halpern}) they replaced
the mapping $T$ with the averaged type mapping $A_{T}$, i.e. with the
mapping:
\begin{equation*}
A_T=(1-\delta)I+\delta T,\quad \delta\in (0,1).
\end{equation*}
Iemoto and Takahashi \cite{Iemoto-Taka} approximated common fixed
points of a  nonexpansive mapping $T$ and of a nonspreading
mapping $S$ in a Hilbert space using Moudafi's iterative scheme
\cite{Moudafi}. They obtained the following Theorem that states
the weak convergence of their iterative method:
\begin{thm}
Let $H$ be a Hilbert space and let $C$ be a nonempty closed and convex subset of $H$. Assume that $Fix(S)\cap Fix(T)\neq\emptyset$.
Define a sequence $(x_n)_{n\in\N}$ as follows:
$$
\left\{
\begin{array}{ll}
x_1 \in C \\
x_{n+1}=(1-\alpha_n)x_n+\alpha_n[\beta_n Sx_n+(1-\beta_n)Tx_n],
\end{array}
\right.
$$
for all $n\in \N$, where $(\alpha_n)_{n\in\N}$, $(\beta_n)_{n\in\N}\subset[0,1]$. Then, the following hold:
\begin{enumerate}
  \item[(i)] If $\displaystyle\liminf_{n\rightarrow\infty}\alpha_n(1-\alpha_n)>0$ and $\displaystyle\sum_{n=1}^{\infty}(1-\beta_n)<\infty$, then $(x_n)_{n\in\N}$ converges weakly to $p\in Fix(S)$;
  \item[(ii)] If $\displaystyle\sum_{n=1}^{\infty}\alpha_n(1-\alpha_n)=\infty$ and $\displaystyle\sum_{n=1}^{\infty}\beta_n<\infty$, then $(x_n)_{n\in\N}$ converges weakly to $p\in Fix(T)$;
  \item[(iii)] If $\displaystyle\liminf_{n\rightarrow\infty}\alpha_n(1-\alpha_n)>0$ and $\displaystyle\liminf_{n\rightarrow\infty}\beta_n(1-\beta_n)>0,$ then $(x_n)_{n\in\N}$ converges weakly to $p\in Fix(S)\cap Fix(T)$.
\end{enumerate}
\end{thm}

In this paper, inspired by Iemoto and Takahashi \cite{Iemoto-Taka}, we introduce an iterative method of Halpern's type to approximate strongly fixed points of a nonexpansive mapping $T$ and a nonspreading
 mapping $S$. A crucial tool to prove the strong convergence of our iterative scheme is the use of averaged type mappings  $A_T$ and $A_S$ which have a regularizing role.\\

\section{Preliminaries}
To begin, we collect some Lemmas which we use in our proofs in the next section.\\
Let $H$ be a real Hilbert space.
\begin{lem}\label{Hilbert}
The following known results hold:
\begin{enumerate}
  \item $\|tx+(1-t)y\|^2=t\|x\|^2+(1-t)\|y\|^2-t(1-t)\|x-y\|^2$,\\  for all $x,y\in H$ and for all $t\in [0,1]$.
  \item $\|x+y\|^2\leq\|x\|^2+2\langle y, x+y \rangle$,\\ for all $x,y\in H$.
\end{enumerate}
\end{lem}
The following Lemma \cite{Taka} characterizes the projection $P_C$.
\begin{lem}\label{proiezione}
 Let $C$ be a closed and convex subset of a real Hilbert space and let $P_C$ be the metric projection from $H$ onto $C$. Given $x\in H$ and $z\in C$; then $z=P_Cx$ if and only if there holds the inequality:
 $$\langle x-z, y-z\rangle \leq 0,\quad \forall y\in C.$$
\end{lem}
%We need also this useful tool.
%\begin{lem}\label{SpHilbert}
%If $(x_n)_{n\in\N}$  is a sequence in $H$ which converges weakly to $z\in H$ then
%  $$\limsup_{n\rightarrow\infty}\|x_n-y\|^2=\limsup_{n\rightarrow\infty}\|x_n-z\|^2+\|z-y\|^2, \quad \forall y\in H.$$
%\end{lem}
To prove our main Theorem, we need some fundamental properties of involved mappings.\\
The following result summarizes some significant properties
of $I-T$ if $T$ is a  nonexpansive mapping (\cite{Byrne},\cite{Goebel}).
\begin{lem} \label{T demich}\label{B=I-T}
Let $C$ be a nonempty closed convex subset of $H$ and let $T:
C\rightarrow C$ be  nonexpansive. Then: \begin{enumerate}
\item $I-T: C\rightarrow H$ is $\frac{1}{2}$-inverse strongly
monotone, i.e.,
$$\frac{1}{2}\|(I-T)x-(I-T)y\|^2\leq \langle x-y, (I-T)x-(I-T)y\rangle,$$
for all $x,y\in C$;
\item moreover, if $Fix(T)\neq\emptyset$, $I-T$ is
demiclosed at $0$, i.e. for every sequence $(x_n)_{n \in \N}$ weakly
convergent to $p$ such that $x_n-Tx_n\rightarrow 0$ as
$n\rightarrow\infty$, it follows $p\in Fix(T)$.
\end{enumerate}
\end{lem}

If $C$ is a nonempty, closed and convex subset of $H$ and $T$ is a
nonlinear mapping of $C$ into itself,  inspired by \cite{Byrne},
we can define the averaged type mapping as follows
\begin{equation}\label{A_T}
A_T=(1-\delta)I+\delta T=I-\delta(I-T)
\end{equation}
where $\delta\in(0,1)$.
We notice that $Fix(T)=Fix(A_T)$ and that if $T$ is a nonexpansive mapping also $A_T$ is nonexpansive.\\

If $S$ is a nonspreading mapping of $C$ into itself and
$Fix(S)\neq\emptyset$, we observe that $A_S$ is
quasi-nonexpansive and further the set of fixed points of $A_S$  is closed and convex.
The following Lemma shows the demiclosedness of $I-S$ at $0$.
\begin{lem}\cite{Iemoto-Taka}\label{demicl S}
Let $C$ be a nonempty, closed and convex subset of $H$. Let $S:
C\rightarrow C$ be a nonspreading mapping such that
$Fix(S)\neq\emptyset$. Then $I-S$ is demiclosed at $0$.
\end{lem}
In the sequel we use the following property of $I-S$.
\begin{lem}\cite{Iemoto-Taka}\label{I-S}
Let $C$ be a nonempty, closed and convex subset of $H$. Let $S: C\rightarrow C$ be a nonspreading mapping. Then
$$\|(I-S)x-(I-S)y\|^2\leq \langle x-y, (I-S)x-(I-S)y\rangle+\frac{1}{2}\bigg(\|x-Sx\|^2+\|y-Sy\|^2\bigg),$$
for all $x,y\in C$.
\end{lem}

If $Fix(S)$ is nonempty, Osilike and Isiogugu \cite{Osilike}
proved that the averaged type mapping $A_S$ is  quasi-firmly type
nonexpansive mapping, i.e. is a firmly type nonexpansive mapping
on fixed points of $S$. On the same line of the proof in  \cite{Osilike}, we prove the following:
\begin{prop}
Let $C$ be a nonempty closed and convex subset of $H$ and let
$S:C\to C$ be a nonspreading mapping such that $Fix(S)$ is
nonempty. Then the averaged type mapping $A_S$
\begin{equation}\label{A_S1} A_S=(1-\delta)I+\delta S,
\end{equation}is quasi-firmly type nonexpansive
mapping with coefficient $k=\left(1-\delta\right)\in (0,1)$.
\end{prop}
%$\textbf{Proof}$\\
\begin{proof}
We obtain
\begin{align*}
\left\|A_Sx-A_Sy\right\|^{2}&=~\left\|\left(1-\delta\right)\left(x-y\right)+\delta\left(Sx-Sy\right)\right\|^{2}\\
\mbox{ (by Lemma \ref{Hilbert}) }&=~\left(1-\delta\right)\left\|x-y\right\|^{2}+\delta\left\|Sx-Sy\right\|^{2}\\  &-~\delta\left(1-\delta\right)\left\|\left(x-Sx\right)-\left(y-Sy\right)\right\|^{2}\\
\mbox{ (by (\ref{S})) }&\leq~\left(1-\delta\right)\left\|x-y\right\|^{2}+\delta\left[\left\|x-y\right\|^{2}+2\left\langle x-Sx,y-Sy\right\rangle\right]\\
&-~\delta\left(1-\delta\right)\left\|\left(x-Sx\right)-\left(y-Sy\right)\right\|^{2}\\ \nonumber
&=~\left\|x-y\right\|^{2}+\frac{2}{\delta}\left\langle \delta\left(x-Sx\right),\delta\left(y-Sy\right)\right\rangle\\
&-~\frac{1-\delta}{\delta}\left\|\delta\left(x-Sx\right)-\delta\left(y-Sy\right)\right\|^{2}\\
\mbox{ (by (\ref{A_S1})) }&=~\left\|x-y\right\|^{2}+\frac{2}{\delta}\left\langle x-A_Sx,y-A_Sy\right\rangle\\
&-~\frac{1-\delta}{\delta}\left\|\left(x-A_Sx\right)-\left(y-A_Sy\right)\right\|^{2}\\
&\leq~\left\|x-y\right\|^{2}+\frac{2}{\delta}\left\langle x-A_Sx,y-A_Sy\right\rangle\\
&-~\left(1-\delta\right)\left\|\left(x-A_Sx\right)-\left(y-A_Sy\right)\right\|^{2}.
\end{align*}
Hence, we have
\begin{equation}\label{A_Sx-A_Sy}
\|A_Sx-A_Sy\|^{2}\leq\|x-y\|^{2}+\frac{2}{\delta}\langle x-A_Sx,y-A_Sy\rangle-(1-\delta)\|(x-A_Sx)-(y-A_Sy)\|^{2}.
\end{equation}
In particular, choosing  in (\ref{A_Sx-A_Sy}) $y=p$, where $p\in
Fix(S)=Fix(A_S)$ we obtain
\begin{equation}\label{A_s quasi-firmly}
\left\|A_Sx-p\right\|^{2}\leq\left\|x-p\right\|^{2}-\left(1-\delta\right)\left\|x-A_Sx\right\|^{2}.
\end{equation}
\end{proof}

A pertinent tool for us is the well-known Lemma of Xu \cite{Xu}.
\begin{lem}\label{Xu}
Let $(a_n)_{n\in\N}$ be a sequence of non-negative real numbers satisfying the following relation:
$$a_{n+1}\leq(1-\alpha_n)a_n+\alpha_n\sigma_n+\gamma_n, \quad n\geq 0,$$
where,
\begin{itemize}
  \item $(\alpha_n)_{n\in\N}\subset [0,1]$, $\displaystyle\sum_{n=1}^\infty \alpha_n=\infty$;
  \item $\displaystyle\limsup_{n\rightarrow\infty}\sigma_n\leq 0$;
  \item $\gamma_n\geq 0$, $\displaystyle\sum_{n=1}^\infty \gamma_n<\infty$.
\end{itemize}
Then, $$\lim_{n\rightarrow\infty}a_n=0.$$
\end{lem}
Finally, a crucial tool for our results is the following Lemma
proved by Maing\'{e}.

\begin{lem}\cite{Mainge}  \label{LemMaing}
Let $(\gamma_n)_{n\in\N}$ be a sequence of real numbers such that there exists a subsequence $(\gamma_{n_j})_{j\in\N}$ of $(\gamma_n)_{n\in\N}$ such that $\gamma_{n_j}<\gamma_{n_j+1}$, for all $j\in\N$. Then, there exists a nondecreasing sequence $(m_k)_{k\in\N}$ of $\N$ such that $\displaystyle \lim_{k\to\infty}m_k=\infty$ and the following properties are satisfied by all (sufficiently large) numbers $k\in\N$:
$$\gamma_{m_k}\leq\gamma_{m_k+1}\quad \mbox{ and } \quad \gamma_k\leq \gamma_{m_k+1}.$$
In fact, $m_k$ is the largest number $n$ in the set $\{1,...,k\}$ such that the condition $\gamma_n<\gamma_{n+1}$ holds.
\end{lem}

\section{The Main result}

\begin{thm}
Let $H$ be a Hilbert space and let $C$ be a nonempty closed and
convex subset of $H$. Let $T:C \to C$ be a nonexpansive mapping
and let $S:C\to C$ be a nonspreading mapping such that $Fix(S)\cap
Fix(T)\neq\emptyset$. Let $A_T$ and $A_S$ be the averaged type mappings,
i.e. $$A_T=(1-\delta)I+\delta T,\,\,A_S=(1-\delta)I+\delta S, \,
\delta \in (0,1).$$ Suppose that $(\alpha_n)_{n\in\N}$ is a real
sequence in $(0,1)$ satisfying the conditions:
\begin{enumerate}
  \item $\displaystyle \lim_{n\rightarrow\infty}\alpha_n=0$,
  \item $\displaystyle \sum_{n=1}^\infty \alpha_n=\infty$.
\end{enumerate}
If  $(\beta_n)_{n\in\N}$ is a sequence in $[0,1]$, we define a
sequence $(x_n)_{n\in\N}$ as follows:
$$
\left\{
\begin{array}{ll}
x_1 \in C \\
x_{n+1}=\alpha_n u+(1-\alpha_n)[\beta_n
A_Tx_n+(1-\beta_n)A_Sx_n],\quad n \in \N.
\end{array}
\right.
$$
Then, the following hold:
\begin{enumerate}
  \item[(i)] If $\displaystyle\sum_{n=1}^{\infty}(1-\beta_n)<\infty$, then $(x_n)_{n\in\N}$ converges strongly to $p\in Fix(T)$;
  \item[(ii)] If $\displaystyle\sum_{n=1}^{\infty}\beta_n<\infty$, then $(x_n)_{n\in\N}$ converges strongly to $p\in Fix(S)$;
  \item[(iii)] If $\displaystyle\liminf_{n\to\infty}\beta_n(1-\beta_n)>0$, then $(x_n)_{n\in\N}$ converges strongly to $p\in Fix(T)\cap Fix(S)$.
\end{enumerate}
\end{thm}
\begin{proof}
We begin to prove that $(x_n)_{n\in\N}$ is bounded.\\
Put
\begin{equation}\label{U}
U_n=\beta_n A_T+(1-\beta_n)A_S.
\end{equation}
Notice that $U_n$ is quasi-nonexpansive, for all $n\in\N$.\\
For $q\in Fix(T)\cap Fix(S)$, we have
\begin{eqnarray}\label{x_n-q}
\nonumber\|x_{n+1}-q\| & = & \|\alpha_n (u-q)+(1-\alpha_n)(U_nx_n-q)\|\\
\nonumber& \leq & \alpha_n \|u-q\|+(1-\alpha_n)\|U_nx_n-q\|\\
& \leq &  \alpha_n \|u-q\|+(1-\alpha_n)\|x_n-q\|
\end{eqnarray}
Since $$\|x_1-q\|\leq \max \{\|u-q\|, \|x_1-q\|\},$$
and by induction we  assume that
$$\|x_n-q\|\leq \max \{\|u-q\|, \|x_1-q\|\},$$
then
\begin{eqnarray*}
\|x_{n+1}-q\| & \leq & \alpha_n \|u-q\|+(1-\alpha_n)\max \{\|u-q\|, \|x_1-q\|\}\\
& \leq & \alpha_n\max \{\|u-q\|, \|x_1-q\|\}+(1-\alpha_n)\max \{\|u-q\|, \|x_1-q\|\}\\
& = & \max \{\|u-q\|, \|x_1-q\|\}.
\end{eqnarray*}
Thus $(x_n)_{n\in\N}$ is bounded. Consequently, $(A_Tx_n)_{n\in\N}$, $(A_Sx_n)_{n\in\N}$  and $(U_nx_n)_{n\in\N}$ are bounded as well.\\
\textbf{Proof of (i)}
We introduce an auxiliary sequence $$z_{n+1}=\alpha_n u+(1-\alpha_n)A_Tx_n, \, n \in \N$$ and we study its properties and the relationship with the sequence $(x_n)_{n \in \N}$.\\
We shall divide the proof into several steps.\\
\textbf{Step 1.} $\displaystyle\lim_{n \to \infty}\|x_n-z_n\|=0$.\\
\textit{Proof of Step 1.}
Observe that
\begin{equation}\label{z(n+1)-A_T}
\lim_{n\rightarrow\infty}\|z_{n+1}-A_Tx_n\|=\lim_{n\rightarrow\infty}\alpha_n\|u-A_Tx_n\|=0.
\end{equation}
Then we get
\begin{eqnarray}\label{x(n+1)-z(n+1)}
\nonumber\|z_{n+1}-x_{n+1}\| & = & \|\alpha_n
u+(1-\alpha_n)A_Tx_n-\alpha_nu-(1-\alpha_n)Ux_n\|\\
\nonumber& = & (1-\alpha_n)\|A_Tx_n-Ux_n\|\\
\nonumber& = &
(1-\alpha_n)\|A_Tx_n-\beta_nA_Tx_n-(1-\beta_n)A_Sx_n\|\\
& = &
(1-\alpha_n)(1-\beta_n)\|A_Tx_n-A_Sx_n\|.
\end{eqnarray}
Since $\displaystyle\sum_{n=1}^{\infty}(1-\beta_n)<\infty$, we have
\begin{equation}\label{x(n)-z(n)}
\lim_{n\rightarrow\infty}\|x_n-z_n\|=0.
\end{equation}
So, also $(z_n)_{n\in\N}$ is bounded.\\
\textbf{Step 2.} $\displaystyle\lim_{n\rightarrow\infty}\|z_n-A_Tz_n\|=0$.\\
\textit{Proof of Step 2.}
We begin to prove that $\displaystyle\lim_{n\rightarrow\infty}\|x_n-Tx_n\|=0$.\\
Let $p\in Fix(T)=Fix(A_T)$. We have
\begin{eqnarray*}
\|z_{n+1}-p\|^2 & = & \|\alpha_nu+(1-\alpha_n)(1-\delta)x_n+(1-\alpha_n)\delta Tx_n -p\|^2\\
& = & \|[(1-\alpha_n)\delta (Tx_n-x_n)+ x_n-p]+\alpha_n(u-x_n)\|^2\\
\mbox{( by Lemma \ref{Hilbert}) }& \leq & \|(1-\alpha_n)\delta (Tx_n-x_n)+ x_n-p\|^2\\ \nonumber
& + & 2\alpha_n\langle u-x_n, z_{n+1}-p\rangle\\
& \leq & (1-\alpha_n)^2 \delta^2\|Tx_n-x_n\|^2+\|x_n-p\|^2\\ \nonumber
& - & 2(1-\alpha_n)\delta\langle x_n-p, x_n-Tx_n \rangle\\
& + & 2\alpha_n \|u-x_n\|\|z_{n+1}-p\|\\
& = & (1-\alpha_n)^2 \delta^2\|x_n-Tx_n\|^2+\|x_n-p\|^2\\ \nonumber
\mbox{ $((I-T)p=0)$ }& - & 2(1-\alpha_n)\delta\langle x_n-p, (I-T)x_n-(I-T)p \rangle\\
& + & 2\alpha_n \|u-x_n\|\|z_{n+1}-p\|\\
\mbox{( by Lemma \ref{B=I-T}) }& \leq & \|x_n-p\|^2+(1-\alpha_n)^2 \delta^2\|x_n-Tx_n\|^2\\ \nonumber
& - & (1-\alpha_n)\delta\|(I-T)x_n-(I-T)p\|^2\\
& + & 2\alpha_n \|u-x_n\|\|z_{n+1}-p\|\\
& = & \|x_n-p\|^2- (1-\alpha_n)\delta[1-\delta(1-\alpha_n)]\|x_n-Tx_n\|^2\\ \nonumber
& + & 2\alpha_n \|u-x_n\|\|z_{n+1}-p\|
\end{eqnarray*}
and hence
\begin{eqnarray*}
& &(1-\alpha_n) \delta[1-\delta(1-\alpha_n)]\|x_n-Tx_n\|^2-2\alpha_n \|u-x_n\|\|z_{n+1}-p\| \\
& \leq & \|x_n-p\|^2-\|z_{n+1}-p\|^2.
\end{eqnarray*}
Set $$L_n=(1-\alpha_n)
\delta[1-\delta(1-\alpha_n)]\|x_n-Tx_n\|^2-2\alpha_n
\|u-x_n\|\|z_{n+1}-p\|;$$
let us consider the following two cases.\\
a) If $L_n\leq 0$, for all $n\geq n_0$ large enough, then
$$(1-\alpha_n)\delta[1-\delta(1-\alpha_n)]\|x_n-Tx_n\|^2\leq2\alpha_n \|u-x_n\|\|z_{n+1}-p\|.$$
So, since $\displaystyle\lim_{n\rightarrow\infty}\alpha_n=0$ and $\displaystyle\lim_{n\rightarrow\infty}(1-\alpha_n)\delta[1-\delta(1-\alpha_n)]=\delta(1-\delta)$,
$$\lim_{n\rightarrow\infty}\|x_n-Tx_n\|=0.$$

b) Assume now that that there exists a subsequence
$(L_{{n_k}})_{k\in\N}$ of $(L_{n})_{n\in\N}$ taking all its
positive terms; so $L_{n_k}> 0$ for every $k \in \N$, then
\begin{equation} 0<L_{n_k}\leq \|x_{n_k}-p\|^2-\|z_{n_{k}+1}-p\|^2. \label{Lnk} \end{equation}
Summing $(\ref{Lnk})$ from $k=1$ to $N$, we obtain
\begin{eqnarray*}
\sum _{k=1}^N L_{n_k} & \leq & \|x_{n_1}-p\|^2+\sum_{k=1}^{N-1}\big(\|x_{n_{k}+1}-p\|^2-\|z_{n_{k}+1}-p\|^2\big)-\|z_{n_{N}+1}-p\|^2\\
& \leq & \|x_{n_1}-p\|^2+\sum_{k=1}^{N-1}\big(\|x_{n_{k}+1}-p\|+\|z_{n_{k}+1}-p\|\big)\|x_{n_{k}+1}-z_{n_{k}+1}\|\\
& \leq & \|x_{n_1}-p\|^2\\
\mbox{ (by (\ref{x(n+1)-z(n+1)})) }& + & \sum_{k=1}^{N-1} (1-\alpha_{n_k})(1-\beta_{n_k})\big(\|x_{n_{k}+1}-p\|+\|z_{n_{k}+1}-p\|\big)\|A_Tx_{n_k}-A_Sx_{n_k}\|\\
& \leq & \|x_{n_1}-p\|^2+ K \sum_{k=1}^{N-1}(1-\beta_{n_k}),
\end{eqnarray*}
where $K=\sup_{k\in\N}\big\{(\|x_{n_{k}+1}-p\|+\|z_{n_{k}+1}-p\|)\|A_Tx_{n_k}-A_Sx_{n_k}\|\big\}$.\\
Since $\displaystyle\sum_{k=1}^\infty(1-\beta_{n_k})<\infty$,
$$\sum _{k=1}^\infty \big((1-\alpha_{n_k})\delta[1-\delta(1-\alpha_{n_k})]\|x_{n_k}-Tx_{n_k}\|^2-2\alpha_{n_k} \|u-x_{n_k}\|\|z_{n_{k}+1}-p\|\big)<\infty.$$
Thus,
$$\lim_{k\rightarrow\infty}\big((1-\alpha_{n_k})\delta[1-\delta(1-\alpha_{n_k})]\|x_{n_k}-Tx_{n_k}\|^2-2\alpha_{n_k} \|u-x_{n_k}\|\|z_{n_{k}+1}-p\|\big)=0,$$
and since $\displaystyle\lim_{k\rightarrow\infty}\alpha_{n_k}=0$ and $\displaystyle\lim_{k\rightarrow\infty}(1-\alpha_{n_k})\delta[1-\delta(1-\alpha_{n_k})]=\delta(1-\delta),$
also in this case we get
$$\lim_{k\rightarrow\infty}\|x_{n_k}-Tx_{n_k}\|=0.$$ Since the remanent terms of the sequence $\|x_{n}-Tx_{n}\|$ are not positive, from the case a) we can conclude that  $$\lim_{n\rightarrow\infty}\|x_{n}-Tx_{n}\|=0.$$
Consequently,
\begin{eqnarray}\label{x(n)-A_Tx(n)}
\lim_{n\rightarrow\infty}\|x_n-A_Tx_n\| & = & \lim_{n\rightarrow\infty}\|x_n-(1-\delta)x_n-\delta Tx_n\|\\ \nonumber
& = & \lim_{n\rightarrow\infty}\delta \|x_n-Tx_n\|=0.
\end{eqnarray}
Furthermore, from $(\ref{x(n)-z(n)})$ and $(\ref{x(n)-A_Tx(n)})$ and
\begin{eqnarray*}
\|z_n-A_Tz_n\| & \leq & \|z_n-x_n\|+\|x_n-A_Tx_n\|+\|A_Tx_n-A_Tz_n\|\\
& \leq &  \|z_n-x_n\|+\|x_n-A_Tx_n\|+\|x_n-z_n\|
\end{eqnarray*}
we get
\begin{equation}\label{z(n)-A_Tz(n)}
\lim_{n\rightarrow\infty}\|z_n-A_Tz_n\|=0.
\end{equation}
\bigskip

\noindent Now, define the real sequence
\begin{equation}\label{tn}t_n=\sqrt{\|z_n-A_Tz_n\|}, \quad n \in \N. \end{equation} Let $z_{t_n}\in C$ be the unique
fixed point of the contraction $V_{t_n}$ defined su $C$ by
\begin{equation}\label{Vt_n}
V_{t_n}x=t_n u+(1-t_n)A_Tx.
\end{equation}
>From  Browder's Theorem \ref{browder}, $\displaystyle
\lim_{n\rightarrow\infty} z_{t_n}=p_0\in Fix(A_T)$; now we prove
that:\\ \textbf{Step 3.}
$\displaystyle\limsup_{n\rightarrow\infty}\langle u-p_0, z_n-p_0\rangle\leq 0$.\\
\textit{Proof of Step 3.} From $(\ref{Vt_n})$, we have
$$z_{t_n}-z_n=t_n(u-z_n)+(1-t_n)(A_Tz_{t_n}-z_n).$$
We compute
\begin{eqnarray*}
\|z_{t_n}-z_n\|^2 & = & \|t_n(u-z_n)+(1-t_n)(A_Tz_{t_n}-z_n)\|^2\\
\mbox{ (by Lemma \ref{Hilbert}) }& \leq & (1-t_n)^2\|A_Tz_{t_n}-z_n\|^2+2t_n \langle u-z_n, z_{t_n}-z_n\rangle\\
& \leq & (1-t_n)^2\big(\|A_Tz_{t_n}-A_Tz_n\|+\|A_Tz_n-z_n\|\big)^2\\
& + & 2t_n \langle u-z_n, z_{t_n}-z_n\rangle\\
& = & (1-t_n)^2\big[\|A_Tz_{t_n}-A_Tz_n\|^2+\|A_Tz_n-z_n\|^2\\
& + & 2\|A_Tz_n-z_n\|\|A_Tz_{t_n}-A_Tz_n\|\big]\\
& + & 2t_n \langle u-z_{t_n}, z_{t_n}-z_n\rangle+2t_n \langle z_{t_n}-z_n, z_{t_n}-z_n\rangle\\
\mbox{ ($A_T$ nonexpansive) }& \leq & (1-t_n)^2\big[\|z_{t_n}-z_n\|^2+\|A_Tz_n-z_n\|^2\\
& + & 2\|A_Tz_n-z_n\|\|z_{t_n}-z_n\|\big]\\
& + & 2t_n \|z_{t_n}-z_n\|^{2}+2t_n \langle u-z_{t_n}, z_{t_n}-z_n\rangle\\
& = & (1+t_n^2)\|z_{t_n}-z_n\|^2\\
& + & \|A_Tz_n-z_n\|\big(\|A_Tz_n-z_n\|+2\|z_{t_n}-z_n\|\big)\\
& + & 2t_n \langle u-z_{t_n}, z_{t_n}-z_n\rangle.
\end{eqnarray*}

Hence
\begin{eqnarray*}
\langle u-z_{t_n}, z_n -z_{t_n}\rangle & \leq & \frac{t_n}{2}\|z_{t_n}-z_n\|^2\\
& + & \frac{\|A_Tz_n-z_n\|}{2t_n}\big(\|A_Tz_n-z_n\|+2\|z_{t_n}-z_n\|\big).
\end{eqnarray*}
>From (\ref{tn}) and by the boundedness of $(z_{t_n})_{n\in\N}$,
$(z_n)_{n\in\N}$ and $(A_Tz_n)_{n\in\N}$  we have
\begin{equation}\label{limsupT(1)}
\limsup_{n\rightarrow\infty}\langle u-z_{t_n}, z_n -z_{t_n} \rangle\leq 0.
\end{equation}
Furthermore,
\begin{eqnarray}\label{limsupT(2)}
\nonumber\langle u-z_{t_n}, z_n -z_{t_n} \rangle & =& \langle
u-p_0, z_n
-z_{t_n} \rangle+\langle p_0-z_{t_n}, z_n -z_{t_n} \rangle\\
 \nonumber& =& \langle u-p_0, z_n -p_0 \rangle+\langle u-p_0, p_0-z_{t_n} \rangle +\langle p_0-z_{t_n}, z_n -z_{t_n} \rangle\\
\end{eqnarray}
Since $\displaystyle \lim_{n\rightarrow\infty} z_{t_n}=p_0\in Fix(A_T)$, we
get
\begin{equation}\label{limsupT(3)}
\lim_{n\rightarrow\infty}\langle p_0-z_{t_n}, z_n -z_{t_n} \rangle
=\lim_{n\rightarrow\infty}\langle u-p_0, p_0-z_{t_n} \rangle=0.
\end{equation}
We conclude from $(\ref{limsupT(1)})$, $(\ref{limsupT(2)})$ and
$(\ref{limsupT(3)})$
\begin{equation}\label{limsupT(5)}
\limsup_{n\rightarrow\infty}\langle u-p_0, z_n -p_0 \rangle\leq 0.
\end{equation}
\textbf{Step 4.} $(z_n)_{n\in\N}$ converges strongly to $p_0\in Fix(T)$.\\
\textit{Proof of Step 4.}
We compute
\begin{eqnarray*}
\|z_{n+1}-p_0\|^2 & = & \|\alpha_n (u-p_0)+(1-\alpha_n)(A_Tx_n-p_0)\|^2\\
\mbox{ (by Lemma \ref{Hilbert}) }& \leq & (1-\alpha_n)^2\|A_Tx_n-p_0\|^2+2\alpha_n\langle u-p_0, z_{n+1} -p_0 \rangle\\
\mbox{ ($A_T$ nonexpansive) }& \leq & (1-\alpha_n)\|x_n-p_0\|^2+2\alpha_n\langle u-p_0, z_{n+1} -p_0 \rangle\\
& \leq & (1-\alpha_n)\big(\|x_n-z_n\|+\|z_n-p_0\|\big)^2\\
& + & 2\alpha_n\langle u-p_0, z_{n+1} -p_0 \rangle\\
& \leq & (1-\alpha_n)\|x_n-z_n\|^2+(1-\alpha_n)\|z_n-p_0\|^2\\
& + & 2(1-\alpha_n)\|x_n-z_n\|\|z_n-p_0\|\\
& + & 2\alpha_n\langle u-p_0, z_{n+1} -p_0 \rangle\\
& \leq & (1-\alpha_n)\|z_n-p_0\|^2\\
\mbox{ (by (\ref{x(n+1)-z(n+1)})) }& + &  (1-\alpha_n)(1-\alpha_{n-1})^2 (1-\beta_{n-1})^2\|A_Tx_{n-1}-A_Sx_{n-1}\|^2\\
& + & 2(1-\alpha_n)(1-\alpha_{n-1}) (1-\beta_{n-1})\|A_Tx_{n-1}-A_Sx_{n-1}\|\|z_n-p_0\|\\
& + & 2\alpha_n\langle u-p_0, z_{n+1} -p_0 \rangle\\ \nonumber
& \leq & (1-\alpha_n)\|z_n-p_0\|^2+  (1-\beta_{n-1})\|A_Tx_{n-1}-A_Sx_{n-1}\|^2\\
& + & 2 (1-\beta_{n-1})\|A_Tx_{n-1}-A_Sx_{n-1}\|\|z_n-p_0\|\\
& + &  2\alpha_n\langle u-p_0, z_{n+1} -p_0 \rangle\\
& \leq & (1-\alpha_n)\|z_n-p_0\|^2+ M (1-\beta_{n-1})\\
& + & 2\alpha_n\langle u-p_0, z_{n+1} -p_0 \rangle,
\end{eqnarray*}
where $M:=\sup_{n\in\N}\big\{\|A_Tx_{n-1}-A_Sx_{n-1}\|^{2}+2\|A_Tx_{n-1}-A_Sx_{n-1}\|\|z_n-p_0\|\big \}$.\\
%Set $\sigma_n=2\langle u-p, z_{n+1} -p \rangle$, $\gamma_n= 2 M (1-\beta_{n-1})$ and $a_n=\|z_n-p\|^2$. Then, we rewrite $(\ref{conv T})$ as
%$$a_{n+1}\leq(1-\alpha_n)a_n+\gamma_n+\alpha_n\sigma_n.$$
Since by hypothesis $\displaystyle\sum_{n=1}^\infty
\alpha_n=\infty$ and $\displaystyle
\sum_{n=1}^\infty(1-\beta_n)<\infty$,  from $(\ref{limsupT(5)})$
we can apply
 Lemma $\ref{Xu}$ and conclude that  $$\lim_{n\rightarrow\infty}\|z_{n+1}-p_0\|=0.$$
 % So, $(z_n)_{n\in\N}$ converges strongly to a fixed point of $T$.
By $\displaystyle \lim_{n \to \infty}\|x_n-z_n\|= 0$, we have
$$\lim_{n\rightarrow\infty}\|x_n-p_0\|=0.$$
Hence, $(x_n)_{n\in\N}$ converges strongly to $p_0\in Fix(T)$.
\end{proof}
\smallskip

\noindent $\textbf{Proof of (ii)}$\\
Again we introduce an other auxiliary sequence
\begin{equation}\label{s(n+1)}
s_{n+1}=\alpha_{n}u+(1-\alpha_{n})A_Sx_{n},
\end{equation}
and we study its properties and the relationship with the sequence
$(x_n)_{n \in \N}$.\\ Recall that $A_S=(1-\delta)I+\delta S$, with $\delta\in(0,1)$.\\
We shall divide the proof into several steps.\\

\begin{proof}
$\textbf{Step 1.}$ $\displaystyle\lim_{n \to \infty}\|x_{n}-s_{n}\|= 0$.\\
\textit{Proof of Step 1.}
We observe that
\begin{equation}\lim_{n\rightarrow\infty}\left\|s_{n+1}-A_Sx_{n}\right\|=\lim_{n\rightarrow\infty}\alpha_{n}\left\|u-A_Sx_{n}\right\|=0. \label{s(n+1)-A_Sx(n)} \end{equation}
We compute
\begin{eqnarray}\label{x(n+1)-s(n+1)}
\nonumber\|x_{n+1}-s_{n+1}\| & = &
\|\alpha_{n}u+(1-\alpha_{n})Ux_{n}-\alpha_{n}u-(1-\alpha_{n})A_Sx_{n}\|\\
\nonumber& = & (1-\alpha_{n})\|Ux_{n}-A_Sx_{n}\|\\ \nonumber & = &
(1-\alpha_{n})\|\beta_{n}A_Tx_{n}+(1-\beta_{n})A_Sx_{n}-A_Sx_{n}\|\\
 & = & (1-\alpha_{n})\beta_{n}\|A_Tx_{n}-A_Sx_{n}\|.
\end{eqnarray}
Since $\sum_{n=1}^\infty\beta_n<\infty$, \begin{equation} \lim_{n\rightarrow\infty}\left\|x_{n}-s_{n}\right\|=0. \label{x(n)-s(n)} \end{equation}
This  shows that also $(s_{n})_{n\in\N}$ is bounded.

$\textbf{Step 2.}$ $\displaystyle\lim_{n\rightarrow\infty}\|x_n-A_Sx_n\|=0$.\\
\textit{Proof of Step 2.}
We begin to prove that $\displaystyle\lim_{n\rightarrow\infty}\|x_n-Sx_n\|=0$.\\

Let $p\in Fix(S)=Fix(A_S)$.
We compute

\begin{eqnarray*}
\|s_{n+1}-p\|^2 & = & \|\alpha_nu+(1-\alpha_n)(1-\delta)x_n+(1-\alpha_n)\delta Sx_n -p\|^2\\
& = & \|[(1-\alpha_n)\delta (Sx_n-x_n)+ x_n-p]+\alpha_n(u-x_n)\|^2\\
\mbox{( by Lemma \ref{Hilbert}) }& \leq & \|(1-\alpha_n)\delta (Sx_n-x_n)+ x_n-p\|^2\\
& + & 2\alpha_n\langle u-x_n, s_{n+1}-p\rangle\\
& \leq & (1-\alpha_n)^2 \delta^2\|Sx_n-x_n\|^2+\|x_n-p\|^2\\
& - & 2(1-\alpha_n)\delta\langle x_n-p, x_n-Sx_n \rangle\\
& + & 2\alpha_n \|u-x_n\|\|s_{n+1}-p\|\\
& = & (1-\alpha_n)^2 \delta^2\|x_n-Sx_n\|^2+\|x_n-p\|^2\\
\mbox{ $((I-S)p=0)$ }& - & 2(1-\alpha_n)\delta\langle x_n-p, (I-S)x_n-(I-S)p \rangle\\
& + & 2\alpha_n \|u-x_n\|\|s_{n+1}-p\|\\
\mbox{( by Lemma \ref{I-S}) }& \leq & \|x_n-p\|^2+(1-\alpha_n)^2 \delta^2\|x_n-Sx_n\|^2\\
& -& 2(1-\alpha_n)\delta\bigg[\|(I-S)x_n-(I-S)p\|^2\\
& - & \frac{1}{2}\bigg(\|x_n-Sx_n\|^2+\|p-Sp\|^2\bigg)\bigg]\\
& + & 2\alpha_n \|u-x_n\|\|s_{n+1}-p\|\\
& = & \|x_n-p\|^2+(1-\alpha_n)^2 \delta^2\|x_n-Sx_n\|^2\\
& - & (1-\alpha_n)\delta \|x_n- Sx_n\|^2+ 2\alpha_n \|u-x_n\|\|s_{n+1}-p\|\\
& = & \|x_n-p\|^2- (1-\alpha_n)\delta[1-\delta(1-\alpha_n)]\|x_n-Sx_n\|^2\\
& + & 2\alpha_n \|u-x_n\|\|s_{n+1}-p\|
\end{eqnarray*}
and hence
\begin{eqnarray*}
& &(1-\alpha_n) \delta[1-\delta(1-\alpha_n)]\|x_n-Sx_n\|^2-2\alpha_n \|u-x_n\|\|s_{n+1}-p\| \\
& \leq & \|x_n-p\|^2-\|s_{n+1}-p\|^2.
\end{eqnarray*}
Set $$L_n=(1-\alpha_n)
\delta[1-\delta(1-\alpha_n)]\|x_n-Sx_n\|^2-2\alpha_n
\|u-x_n\|\|s_{n+1}-p\|;$$
let us consider the following two cases.\\
a) If $L_n\leq 0$, for all $n\geq n_0$ large enough, then
$$(1-\alpha_n)\delta[1-\delta(1-\alpha_n)]\|x_n-Sx_n\|^2\leq2\alpha_n \|u-x_n\|\|s_{n+1}-p\|.$$
So, since $\displaystyle\lim_{n\rightarrow\infty}\alpha_n=0$ and $\displaystyle\lim_{n\rightarrow\infty}(1-\alpha_n)\delta[1-\delta(1-\alpha_n)]=\delta(1-\delta)$,
$$\lim_{n\rightarrow\infty}\|x_n-Sx_n\|=0.$$

b) Assume now that there exists a subsequence
$(L_{{n_k}})_{k\in\N}$ of $(L_{n})_{n\in\N}$ taking all its
positive terms; so $L_{n_k}> 0$ for every $k \in \N$, then
\begin{equation}
 0<L_{n_k}\leq \|x_{n_k}-p\|^2-\|s_{n_{k}+1}-p\|^2. \label{LnkS}
\end{equation}
Summing $(\ref{LnkS})$ from $k=1$ to $N$, we obtain
\begin{eqnarray*}
\sum _{k=1}^N L_{n_k} & \leq & \|x_{n_1}-p\|^2+\sum_{k=1}^{N-1}\big(\|x_{n_{k}+1}-p\|^2-\|s_{n_k+1}-p\|^2\big)-\|s_{n_{N}+1}-p\|^2\\
& \leq & \|x_{n_1}-p\|^2+\sum_{k=1}^{N-1}\big(\|x_{n_k+1}-p\|+\|s_{n_k+1}-p\|\big)\|x_{n_k+1}-s_{n_k+1}\|\\
& \leq & \|x_{n_1}-p\|^2\\
\mbox{ (by(\ref{x(n+1)-s(n+1)})) }& + & \sum_{k=1}^{N-1} (1-\alpha_{n_k})\beta_{n_k}\big(\|x_{n_k+1}-p\|+\|s_{n_k+1}-p\|\big)\|A_Tx_{n_k}-A_Sx_{n_k}\|\\
& \leq & \|x_{n_1}-p\|^2+ K \sum_{k=1}^{N-1}\beta_{n_k},
\end{eqnarray*}
where $K=\sup_{k\in\N}\big\{(\|x_{n_k+1}-p\|+\|s_{n_k+1}-p\|)\|A_Tx_{n_k}-A_Sx_{n_k}\|\big\}$.\\
Since $\displaystyle\sum_{k=1}^\infty\beta_{n_k}<\infty$,
$$\sum _{k=1}^\infty \big((1-\alpha_{n_k})\delta[1-\delta(1-\alpha_{n_k})]\|x_{n_k}-Sx_{n_k}\|^2-2\alpha_{n_k} \|u-x_{n_k}\|\|s_{n_k+1}-p\|\big)<\infty.$$
Thus,
$$\lim_{k\rightarrow\infty}\big((1-\alpha_{n_k})\delta[1-\delta(1-\alpha_{n_k})]\|x_{n_k}-Sx_{n_k}\|^2-2\alpha_{n_k} \|u-x_{n_k}\|\|s_{n_k+1}-p\|\big)=0,$$
and since $\displaystyle\lim_{k\rightarrow\infty}\alpha_{n_k}=0$ and $\displaystyle\lim_{k\rightarrow\infty}(1-\alpha_{n_k})\delta[1-\delta(1-\alpha_{n_k})]=\delta(1-\delta),$
also in this case we get
$$\lim_{k\rightarrow\infty}\|x_{n_k}-Sx_{n_k}\|=0.$$ As in i), we can conclude that  $$\lim_{n\rightarrow\infty}\|x_{n}-Sx_{n}\|=0.$$
Consequently,
\begin{eqnarray}\label{x(n)-A_Sx(n)}
\nonumber\lim_{n\rightarrow\infty}\|x_n-A_Sx_n\| & = & \lim_{n\rightarrow\infty}\|x_n-(1-\delta)x_n-\delta Sx_n\|\\
& = & \lim_{n\rightarrow\infty}\delta \|x_n-Sx_n\|=0.
\end{eqnarray}
Moreover, from $(\ref{x(n)-s(n)})$ and $(\ref{x(n)-A_Sx(n)})$,
\begin{equation*}
\lim_{n\rightarrow\infty}\|s_n-A_Ss_n\|=\lim_{n\rightarrow\infty}\delta\|s_n-Ss_n\|=0.
\end{equation*}

$\textbf{Step 3.}$ $\displaystyle\limsup_{n\rightarrow\infty}\left\langle u-P_{Fix(S)}u,s_{n}-P_{Fix(S)}u\right\rangle\leq0$.\\
\textit{Proof of Step 4.} We may assume without loss of generality
that there exists a subsequence $(s_{n_{j}})_{j\in\N}$ of
$(s_{n})_{n\in\N}$ such that $s_{n_{j}}\rightharpoonup v$ and
\begin{eqnarray*}\limsup_{n\rightarrow\infty}\left\langle
u-P_{Fix(S)}u,s_{n}-P_{Fix(S)}u\right\rangle&=&\lim_{j\rightarrow\infty}\left\langle
u-P_{Fix(S)}u,s_{n_{j}}-P_{Fix(S)}u\right\rangle\\&=&\left\langle
u-P_{Fix(S)}u,v-P_{Fix(S)}u\right\rangle\end{eqnarray*} Since
$\displaystyle\lim_{n\rightarrow\infty}\|s_n-Ss_n\|=0$ and from
$I-S$ is demiclosed at $0$, $v\in Fix(S)=Fix(A_S)$. Then by
$(\ref{proiezione})$, we have
\begin{equation}\label{limsup S}\limsup_{n\rightarrow\infty}\left\langle u-P_{Fix(S)}u,s_{n}-P_{Fix(S)}u\right\rangle=\left\langle
u-P_{Fix(S)}u,v-P_{Fix(S)}u\right\rangle\leq0.
\end{equation}

$\textbf{Step 4.}$  $(s_n)_{n\in\N}$ converges strongly to $P_{Fix(S)}u$.\\
\textit{Proof of Step 5.}
We compute
\begin{eqnarray*}
\|s_{n+1}-P_{Fix(S)}u\|^2 & = & \|\alpha_n (u-P_{Fix(S)}u)+(1-\alpha_n)(A_Sx_n-P_{Fix(S)}u)\|^2\\
\mbox{ (by Lemma \ref{Hilbert}) }& \leq & (1-\alpha_n)^2\|A_Sx_n-P_{Fix(S)}u\|^2\\
& + &2\alpha_n\langle u- P_{Fix(S)}u, s_{n+1} -P_{Fix(S)}u \rangle\\
\mbox{ ($A_S$ nonexpansive) }& \leq & (1-\alpha_n)\|x_n-P_{Fix(S)}u\|^2\\
& + &2\alpha_n\langle u- P_{Fix(S)}u, s_{n+1} -P_{Fix(S)}u \rangle\\
& \leq & (1-\alpha_n)\big(\|x_n-s_n\|+\|s_n-P_{Fix(S)}u\|\big)^2\\
& + &2\alpha_n\langle u- P_{Fix(S)}u, s_{n+1} -P_{Fix(S)}u \rangle\\
& \leq & (1-\alpha_n)\|x_n-s_n\|^2+(1-\alpha_n)\|s_n-P_{Fix(S)}u\|^2\\
& + & 2(1-\alpha_n)\|x_n-s_n\|\|s_n-P_{Fix(S)}u\|\\
& + & 2\alpha_n\langle u-P_{Fix(S)}u, s_{n+1} -P_{Fix(S)}u \rangle\\
& \leq & (1-\alpha_n)\|s_n-P_{Fix(S)}u\|^2\\
\mbox{ by (\ref{x(n+1)-s(n+1)}) }& + &  (1-\alpha_n)(1-\alpha_{n-1})^2 \beta_{n-1}^2\|A_Tx_{n-1}-A_Sx_{n-1}\|^2\\
& + & 2(1-\alpha_n)(1-\alpha_{n-1}) \beta_{n-1}\|A_Tx_{n-1}-A_Sx_{n-1}\|\|s_n-P_{Fix(S)}u\|\\
 & + & 2\alpha_n\langle u-P_{Fix(S)}u, s_{n+1} -P_{Fix(S)}u \rangle\\
& \leq & (1-\alpha_n)\|s_n-P_{Fix(S)}u\|^2+  \beta_{n-1}\|A_Tx_{n-1}-A_Sx_{n-1}\|^2\\
& + & 2 \beta_{n-1}\|A_Tx_{n-1}-A_Sx_{n-1}\|\|s_n-P_{Fix(S)}u\|\\
& + &  2\alpha_n\langle u-P_{Fix(S)}u, s_{n+1} -P_{Fix(S)}u \rangle\\
& \leq & (1-\alpha_n)\|s_n-P_{Fix(S)}u\|^2+ M \beta_{n-1}\\
& + & 2\alpha_n\langle u-P_{Fix(S)}u, s_{n+1} -P_{Fix(S)}u \rangle,
\end{eqnarray*}
where $M:=\sup_{n\in\N}\big\{\|A_Tx_{n-1}-A_Sx_{n-1}\|^2+2\|A_Tx_{n-1}-A_Sx_{n-1}\|\|s_n-P_{Fix(S)}u\|\big \}$.\\
%Set $\sigma_n=2\langle u-Pu, s_{n+1} -Pu \rangle$, $\gamma_n= 2 M \beta_{n-1}$ and $a_n=\|s_n-Pu\|^2$. Then, we rewrite $(\ref{conv T})$ as
%$$a_{n+1}\leq(1-\alpha_n)a_n+\gamma_n+\alpha_n\sigma_n.$$
Since $\displaystyle\sum_{n=1}^\infty \alpha_n=\infty$ and $\displaystyle \sum_{n=1}^\infty\beta_n<\infty$ and from $(\ref{limsup S})$  we can apply
 Lemma $\ref{Xu}$ and we conclude that  $$\lim_{n\rightarrow\infty}\|s_{n+1}-P_{Fix(S)}u\|=0.$$ So, $(s_n)_{n\in\N}$ converges strongly to $P_{Fix(S)}u\in Fix(S)$.
Since $\displaystyle\lim_{n \to \infty}\|x_n-s_n\|= 0$, we have
$$\lim_{n\rightarrow\infty}\|x_n-P_{Fix(S)}u\|=0,$$
i.e.  $(x_n)_{n\in\N}$ converges strongly to $P_{Fix(S)}u\in
Fix(S)$.
\end{proof}

$\textbf{Proof of (iii)}$ \\
\begin{proof}
Let $q\in Fix(S)\cap Fix(T)$.\\
Since in this last case, the techniques used in i) and ii) fail, we turn our attention on the monotony of the sequence $(\|x_n-q\|)_{n \in \N}$.
We consider the following two cases.\\
\begin{enumerate}
\item [\textbf{Case 1.}]$\|x_{n+1}-q\|\leq\|x_n-q\|, \mbox{ for every } n\geq n_0$ large enough.\\
\item [\textbf{Case 2.}] There exists a subsequence
$(\|x_{n_j}-q\|)_{j\in\N}$ of $(\|x_{n}-q\|)_{n\in\N}$
 such that $$\|x_{n_j}-q\|<\|x_{n_j+1}-q\| \mbox{ for all
 }j\in\N.$$
\end{enumerate}
\smallskip

\noindent

 \begin{enumerate}
\item [\textbf{Case 1.}]
$\displaystyle\lim_{n\rightarrow\infty}\|x_n-q\|$ exists finite
and hence
\begin{equation}\label{(x_n+1-q)-(x_n-q)}
\lim_{n\rightarrow\infty}(\|x_{n+1}-q\|-\|x_n-q\|)=0.
\end{equation}
We shall divide the proof into several steps.\\
$\textbf{Step 1.}$   $\displaystyle\lim_{n\rightarrow\infty}\|x_n-A_Sx_n\|=0$.\\
\textit{Proof of Step 1.}
Consider
\begin{equation}\label{x(n+1)}
x_{n+1}=\alpha_n u+(1-\alpha_n)\beta_n
\big(A_Tx_n+(1-\beta_n)A_Sx_n\big).
\end{equation} We compute
\begin{eqnarray*}
\|\beta_n(A_Tx_n-q)+(1-\beta_n)(A_Sx_n-q)\|^2 & = & \beta_n\|A_Tx_n-q\|^2\\
& + & (1-\beta_n)\|A_Sx_n-q\|^2\\
& - & \beta_n(1-\beta_n)\|A_Tx_n-A_Sx_n\|^2\\
\mbox{ ($A_T$ nonexpansive and by (\ref{A_s quasi-firmly})) }& \leq & \beta_n\|x_n-q\|^2+(1-\beta_n)\|x_n-q\|^2\\
& - & (1-\beta_n)(1-\delta)\|x_n-A_Sx_n\|^2\\
& - & \beta_n(1-\beta_n)\|A_Tx_n-A_Sx_n\|^2\\
& = & \|x_n-q\|^2-(1-\beta_n)(1-\delta)\|x_n-A_Sx_n\|^2\\
& - & \beta_n(1-\beta_n)\|A_Tx_n-A_Sx_n\|^2
\end{eqnarray*}
We recall that $U_n=\beta_n A_T+(1-\beta_n)A_S$. \\So, we get
\begin{equation}\label{Ux(n)-q}
\|U_nx_n-q\|^2
\leq\|x_n-q\|^2-(1-\beta_n)(1-\delta)\|x_n-A_Sx_n\|^2-\beta_n(1-\beta_n)\|A_Tx_n-A_Sx_n\|^2.
\end{equation}
We have
\begin{eqnarray}\label{x(n+1)-q}
\nonumber\|x_{n+1}-q\|^2 & = & \|U_nx_n-q+\alpha_n(u-U_nx_n)\|^2\\
 \nonumber& \leq & \|U_nx_n-q\|^2+\alpha_n(\alpha_n \|u-U_nx_n\|^2+2\|U_nx_n-q\|\|u-U_nx_n\|)\\
\nonumber\mbox{ (by (\ref{Ux(n)-q})) }& \leq & \|x_n-q\|^2-(1-\beta_n)(1-\delta)\|x_n-A_Sx_n\|^2\\
& - & \beta_n(1-\beta_n)\|A_Tx_n-A_Sx_n\|^2+\alpha_n M,
\end{eqnarray}
where $M:=\displaystyle\sup_{n\in\N}\big\{\alpha_n
\|u-U_nx_n\|^2+2\|U_nx_n-q\|\|u-U_nx_n\|\big\}$. From
$(\ref{x(n+1)-q})$, we derive
$$\|x_{n+1}-q\|^2 \leq \|x_n-q\|^2-(1-\beta_n)(1-\delta)\|x_n-A_Sx_n\|^2+\alpha_n M,$$
hence
\begin{equation}\label{song}
(1-\beta_n)(1-\delta)\|x_n-A_Sx_n\|^2\leq\|x_n-q\|^2-\|x_{n+1}-q\|^2+\alpha_n M.
\end{equation}
>From (\ref{(x_n+1-q)-(x_n-q)}) and  $\displaystyle \lim_{n\rightarrow\infty}\alpha_n=0$, we get
$$\lim_{n \to \infty}((1-\beta_n)(1-\delta)\|x_n-A_Sx_n\|^2)=0.$$
Since $\displaystyle \liminf_{n\to \infty}\beta_n(1-\beta_n)>0$, we have
\begin{equation}
\lim_{n\rightarrow\infty}\|x_{n}-A_Sx_{n}\|=\lim_{n\rightarrow\infty} \delta \|x_n-Sx_n\|=0. \label{U_n
x(n)-A_Sx(n)} \end{equation}

$\textbf{Step 2.}$   $\displaystyle\lim_{n \to \infty}\|A_Tx_n-A_Sx_n\|= 0$.\\
\textit{Proof of Step 2.}
Moreover, from $(\ref{x(n+1)-q})$, we also can derive
$$\|x_{n+1}-q\|^2  \leq  \|x_n-q\|^2-\beta_n(1-\beta_n)\|A_Tx_n-A_Sx_n\|^2+\alpha_n M,$$
hence
\begin{equation*}
\beta_n(1-\beta_n)\|A_Tx_n-A_Sx_n\|^2\leq\|x_n-q\|^2-\|x_{n+1}-q\|^2+\alpha_n M.
\end{equation*}
As above,  we can conclude that\\
\begin{equation}\label{A_T-A_S}
\lim_{n\rightarrow\infty}\|A_Tx_n-A_Sx_n\|=\lim_{n\rightarrow\infty}\delta\|Tx_n-Sx_n\|=0.
\end{equation}
>From $(\ref{U_n x(n)-A_Sx(n)})$ and $(\ref{A_T-A_S})$, it  follows that
\begin{equation}\label{xn-Txn U}
\lim_{n\to\infty}\|x_n-Tx_n\|=0.
\end{equation}
Let $F=Fix(T)\cap Fix(S)$. \\
$\textbf{Step 3.}$ $\limsup_{n\rightarrow\infty}\left\langle
u-P_Fu,x_{n}-P_Fu\right\rangle\leq0$.\\
\textit{Proof of Step 3.} We may assume without loss of generality
that there exists a subsequence $(x_{n_{j}})_{j\in\N}$ of
$(x_{n})_{n\in\N}$ such that $x_{n_{j}}\rightharpoonup v$ and
\begin{eqnarray}\label{limsup U}
\nonumber\limsup_{n\rightarrow\infty}\left\langle
u-P_{F}u,x_{n}-P_{F}u\right\rangle & = &
\lim_{j\rightarrow\infty}\left\langle
u-P_{F}u,x_{n_{j}}-P_{F}u\right\rangle \\ & = & \left\langle
u-P_{F}u,v-P_{F}u\right\rangle.
\end{eqnarray}
By $(\ref{xn-Txn U})$ and (\ref{U_n
x(n)-A_Sx(n)}) and by the demiclosedness of
$I-T$ at $0$ and of $I-S$ at $0$, $v\in F=Fix(T)\cap Fix(S)$. Then we can conclude
that
\begin{eqnarray*}
\nonumber\limsup_{n\rightarrow\infty}\left\langle
u-P_{F}u,x_{n}-P_{F}u\right\rangle &=&\left\langle
u-P_{F}u,v-P_{F}u\right\rangle\leq 0.
\end{eqnarray*}

$\textbf{Step 4.}$  $(x_n)_{n\in\N}$ converges strongly to $P_{F}u$.\\
\textit{Proof of Step 4. } We compute
\begin{eqnarray}\label{x_n Xu}
\nonumber \|x_{n+1}-P_{F}u\|^2 & =& \|\alpha_n (u-P_{F}u)+(1-\alpha_n)(U_nx_n-P_{F}u)\|^2\\
\nonumber \mbox{ (by Lemma \ref{Hilbert}) }& \leq & (1-\alpha_n)^2\|U_nx_n-P_{F}u\|^2\\
\nonumber & + & 2\alpha_n\langle u-P_{F}u, x_{n+1} -P_{F}u \rangle\\
\nonumber \mbox{ ($U_n$ quasi-nonexpansive) }& \leq & (1-\alpha_n)\|x_n-P_{F}u\|^2\\
& + & 2\alpha_n\langle u-P_{F}u, x_{n+1} -P_{F}u \rangle.
\end{eqnarray}
%Set $\sigma_n=2\langle u-Pu, x_{n+1} -Pu \rangle$ and $a_n=\|x_n-Pu\|^2$. Then, we rewrite $(\ref{conv x(n)})$ as
%$$a_{n+1}\leq(1-\alpha_n)a_n+\alpha_n\sigma_n.$$
Since  $\displaystyle\sum_{n=1}^\infty \alpha_n=\infty$, we can apply
 Lemma $\ref{Xu}$ and conclude that  $$\lim_{n\rightarrow\infty}\|x_{n+1}-P_{F}u\|=0.$$
Finally, $(x_n)_{n\in\N}$ converges strongly to $P_{F}u$.\\
\smallskip

\noindent

\item [\textbf{Case 2.}] Let $q=P_Fu$. Then there exists a subsequence
$(\|x_{n_j}-P_Fu\|)_{j\in\N}$ of $(\|x_{n}-P_Fu\|)_{n\in\N}$
 such that $$\|x_{n_j}-P_Fu\|<\|x_{n_j+1}-P_Fu\|\mbox{ for all }j\in\N.$$ By Lemma \ref{LemMaing},
 there exists a strictly increasing sequence $(m_k)_{k\in\N}$ of positive integers such that $\displaystyle\lim_{k \to \infty}m_k=+\infty$ and the following properties are satisfied by all numbers $k\in\N$:
 \begin{equation}\label{LemMaing}
 \|x_{m_k}-P_Fu\|\leq\|x_{m_k+1}-P_Fu\| \quad \mbox{and} \quad \|x_k-P_Fu\|\leq\|x_{m_k+1}-P_Fu\|.
 \end{equation}
Consequently,
\begin{eqnarray*}
\nonumber 0 & \leq & \lim_{k \to \infty}\big(\|x_{m_k+1}-P_Fu\|-\|x_{m_k}-P_Fu\|\big)\\
\nonumber & \leq & \limsup_{n \to \infty}\big(\|x_{n+1}-P_Fu\|-\|x_{n}-P_Fu\|\big)\\
\nonumber \mbox{ (by (\ref{x_n-q})) } & \leq & \limsup_{n\to \infty}\big( \alpha_n\|u-P_Fu\|+(1-\alpha_n)\|x_n-P_Fu\|-\|x_n-P_Fu\|\big)\\
\mbox{ $(\alpha_n\rightarrow 0)$ }&  = &  \limsup_{n\to \infty} \alpha_n\big(\|u-P_Fu\|-\|x_n-P_Fu\|\big)=0.
\end{eqnarray*}
Hence,
\begin{equation}\label{(xn+1-q)-(x_n-q)2}
\lim_{k \to \infty}\big(\|x_{m_k+1}-P_Fu\|-\|x_{m_k}-P_Fu\|\big)=0.
\end{equation}
As in the \textbf{Case 1.}, we can prove that
$$
\lim_{k\rightarrow\infty}\|x_{m_k}-Sx_{m_k}\|=\lim_{k\rightarrow\infty}\|x_{m_k}-Tx_{m_k}\|=0$$
and by the demiclosedness of $I-T$ at $0$ and of $I-S$ at $0$, we obtain that
\begin{equation}\label{x_m Xu}\limsup_{k\rightarrow\infty}\left\langle
u-P_Fu,x_{m_k}-P_Fu\right\rangle\leq0.\end{equation}  We replace in
(\ref{x_n Xu})  $n$ with $m_k$, then
\begin{equation*}
\|x_{m_k+1}-P_Fu\|^2\leq (1-\alpha_{m_k})\|x_{m_k}-P_Fu\|^2+2\alpha_{m_k}\langle u-P_Fu, x_{m_k+1}-P_Fu\rangle.
\end{equation*}
In particular, we get
\begin{eqnarray}\label{x_mk-P_Fu}
\nonumber\alpha_{m_k}\|x_{m_k}-P_Fu\|^2 & \leq &  \|x_{m_k}-P_Fu\|^2-\|x_{m_k+1}-P_Fu\|^2\\
\nonumber & + & 2\alpha_{m_k}\langle u-P_Fu, x_{m_k+1}-P_Fu\rangle\\
\mbox{ (by (\ref{LemMaing})) }& \leq & 2\alpha_{m_k}\langle u-P_Fu, x_{m_k+1}-P_Fu\rangle.
\end{eqnarray}
Then, from (\ref{x_m Xu}), we obtain
\begin{equation*}
\limsup_{k\to\infty}\|x_{m_k}-P_Fu\|^2 \leq 2\limsup_{k\to\infty}\langle u-P_Fu, x_{m_k+1}-P_Fu\rangle\leq 0.
\end{equation*}
Thus, from (\ref{LemMaing}) and (\ref{(xn+1-q)-(x_n-q)2}), we conclude that
$$\limsup_{k\to\infty}\|x_{k}-P_Fu\|^2\leq\limsup_{k\to\infty}\|x_{m_k+1}-P_Fu\|^2=0,$$
i.e., $(x_n)_{n\in\N}$ converges strongly to $P_F u$.
\end{enumerate}
\end{proof}

\begin{rem}
The inequality (\ref{song}) plays a crucial role in the proof of
$(iii)$ as the similar inequality (3.3) in Theorem 3.1  of
\cite{Song-Chai}. In fact in both proofs by these inequalities some important properties of the sequence follow.\\
We remark that our tools are different from theirs because the
techniques used in \cite{Song-Chai} seem questionable.
\end{rem}

% ----------------------------------------------------------------

\end{document}